\documentclass[11pt]{amsart}
\usepackage[margin=1in]{geometry}

\usepackage{amssymb}
\usepackage{amsthm}
\usepackage{amsmath}
\usepackage{mathrsfs}
\usepackage{amsbsy}
\usepackage[all]{xy}
\usepackage{bm}
\usepackage{hyperref}
\usepackage{tikz}
\usepackage{array}
\usepackage{float}
\usepackage{enumerate}
\usepackage{xcolor}
\usepackage{hhline}
\allowdisplaybreaks
\usepackage[noadjust]{cite}
\usetikzlibrary{calc,decorations.markings,intersections}
\usetikzlibrary{decorations.pathreplacing}
\usepackage{pgfplots}
\pgfplotsset{compat=1.15}

\usepackage{caption}
\usepackage{tabu}
\usepackage{diagbox}
\usepackage{enumerate, amsfonts, mathtools, thmtools,xparse,colortbl,url,hypcap,tikz,xspace,accents,enumitem,letltxmacro,comment,bbm,ytableau}
\usepackage[procnames]{listings}
\usepackage[colorinlistoftodos]{todonotes}
\usetikzlibrary{calc,through,backgrounds,shapes,matrix,math,trees}

\usepackage[noabbrev,capitalise]{cleveref}

\newenvironment{enumerate*}
  {\begin{enumerate}[(I)]
    \setlength{\itemsep}{10pt}
    \setlength{\parskip}{0pt}}
  {\end{enumerate}}

\newtheorem{theorem}{Theorem}[section]

\newtheorem{corollary}[theorem]{Corollary}
\newtheorem{conjecture}[theorem]{Conjecture}

\newtheorem{lemma}[theorem]{Lemma}

\theoremstyle{definition}

\title[Hermite--Hadamard inequalities for nearly-spherical domains]{Hermite--Hadamard inequalities\\ for nearly-spherical domains}

\author[]{Noah Kravitz}
\address[]{Department of Mathematics, Princeton University, Princeton, NJ 08540, USA}
\email{nkravitz@princeton.edu}

\author[]{Mitchell Lee}
\address[]{Unaffiliated, Cambridge, MA 02140, USA}
\email{trivial171@gmail.com}

\begin{document}
\maketitle

\begin{abstract}
A conjecture of Pasteczka, generalizing the classical Hermite--Hadamard Inequality, states that if $\Omega \subseteq \mathbb{R}^d$ is a compact convex domain such that $\Omega$ and $\partial \Omega$ have the same center of mass, then for every convex function $f: \Omega \to \mathbb{R}^d$, the average value of $f$ on $\Omega$ is less than or equal to the average value of $f$ on $\partial \Omega$. Pasteczka proved this conjecture for the case where $\Omega$ is a polytope with an inscribed ball. We generalize this result by proving Pasteczka's conjecture in the case where some point lies at most $(d+1)|\Omega|/|\partial \Omega|$ away from all hyperplanes tangent to $\partial \Omega$.
\end{abstract}

\section{Introduction}
The well-known Hermite--Hadamard Inequality \cite{hermite, hadamard} states that if $f: [a,b] \to \mathbb{R}$ is a convex function, then
$$\frac{1}{b-a}\int_a^b f(x) \,dx \leq \frac{1}{2}(f(a)+f(b))$$
and equality holds if and only if $f$ is affine-linear. Intuitively, a convex function assumes a larger average value on the boundary of $[a,b]$ than on the interior because it ``curves up'' closer to the boundary. Our goal is to generalize this inequality to higher dimensions.

As a first pass, let $\Omega \subseteq \mathbb{R}^d$ be a compact convex domain (our higher-dimensional analogue of the interval $[a,b]$).  We say that $\Omega$ is a \emph{Jensen domain} if 
\begin{equation}\label{eq:main}
\frac{1}{|\Omega|} \int_\Omega f \,dV \leq \frac{1}{|\partial \Omega|} \int_{\partial \Omega}f \,d\sigma
\end{equation}
for every convex function $f: \Omega \to \mathbb{R}$.  One might hope that every compact convex domain is Jensen, but this hope is too optimistic.  Indeed, Pasteczka \cite{pasteczka} noticed that if the center of mass of $\Omega$ does not coincide with the center of mass of $\partial \Omega$, then some linear function $f$ provides a counterexample to \eqref{eq:main}: When $f$ is linear, the left- and right-hand sides of \eqref{eq:main} are the values of $f$ at the centers of mass of $\Omega$ and $\partial \Omega$, respectively, and $f$ can be chosen so that the former is larger than the latter.

With this example in mind, we say that a compact convex domain $\Omega \subseteq \mathbb{R}^d$ is a \emph{Jensen candidate} if $\Omega$ and $\partial \Omega$ have the same center of mass (equivalently, if \eqref{eq:main} holds for all affine-linear functions $f$).  In $1$ dimension, every compact convex set is an interval and hence a Jensen candidate, so, in hindsight, it is not surprising that this condition should be an ingredient in a higher-dimensional generalization of Hermite--Hadamard.  Pasteczka conjectured that being a Jensen candidate is not only necessary but also sufficient for being a Jensen domain.

\begin{conjecture}[Pasteczka \cite{pasteczka}]
Every Jensen candidate is also a Jensen domain.
\end{conjecture}

As Pasteczka observed in his original paper \cite{pasteczka}, it is easy to see that balls are Jensen domains, and it follows quickly from the $1$-dimensional Hermite--Hadamard inequality that parallelopipeds are Jensen domains.  Pasteczka also proved his conjecture for convex polytopes with inscribed balls.

\begin{theorem}[Pastezcka \cite{pasteczka}]\label{thm:pasteczka}
Let $\Omega \subseteq \mathbb{R}^d$ ($d \geq 2$) be a convex polytope with an inscribed $d$-ball. If $\Omega$ is a Jensen candidate, then it is a Jensen domain.
\end{theorem}

Shortly after the present article first appeared as a preprint, Nazarov gave a beautiful resolution of Pasteczka's conjecture in an unpublished comment on MathOverflow \cite{fedja}.  In particular, he proved that the conjecture holds in two dimensions, and he provided a counterexample for three and more dimensions.  This counterexample involves a domain $\Omega$ that is very long and skinny.

We mention that other higher-dimensional analogues of the Hermite--Hadamard Inequality have received substantial attention.  Several authors \cite{stefan1,BBHLLSS,larson} have established the inequality \eqref{eq:main} up to a constant factor if $f$ is nonnegative (even without the assumption that $\Omega$ is a Jensen candidate). Similar questions have been explored for general subharmonic functions $f$ \cite{MN,NP,stefan1,LS}.  Another popular direction of inquiry ((e.g., \cite{CC,CCE, chen, bessenyei})), motivated by connections with Choquet theory, has concerned inequalities similar to \eqref{eq:main} but with $d\sigma$ on the right-hand side replaced by a non-uniform measure on $\partial \Omega$.  See \cite{DP} for further references to older work.

The purpose of the present note is to establish Pasteczka's conjecture for a relatively large class of domains, including all domains that are ``nearly spherical''.  This result can be understood as complementing Nazarov's counterexample: Pasteczka's conjecture holds for domains that are not very oblong but can fail for domains that are very oblong.

\begin{figure}[h]
    \centering
\begin{tikzpicture}[scale=1.5,
  dot/.style={draw,fill,circle,inner sep=1pt},
  tangent/.style={
        decoration={
            markings,
            mark=
                at position #1
                with
                {
                    \coordinate (tangent point-\pgfkeysvalueof{/pgf/decoration/mark info/sequence number}) at (0pt,0pt);
                    \coordinate (tangent unit vector-\pgfkeysvalueof{/pgf/decoration/mark info/sequence number}) at (1,0pt);
                    \coordinate (tangent orthogonal unit vector-\pgfkeysvalueof{/pgf/decoration/mark info/sequence number}) at (0pt,1);
                }
        },
        postaction=decorate
    },
    use tangent/.style={
        shift=(tangent point-#1),
        x=(tangent unit vector-#1),
        y=(tangent orthogonal unit vector-#1)
    },
    use tangent/.default=1
  ]
  \def\a{4} 
  \def\b{2} 
  \draw[thick, tangent=0.35, scale = 0.72] (0,0) ellipse ({\a} and {\b});
  \coordinate (o) at (0,0);
  \draw[use tangent] (2.5,0) coordinate (x) -- (-3,0) coordinate (y);
  \coordinate (foot) at ($(x)!(o)!(y)$);
  \draw (o) -- (foot);
  \draw[use tangent] (foot) --++(0.25,0) --++ (0,0.25) --++(-0.25,0);
  
  \fill (tangent point-1) circle [radius=2pt] node[above left] {$x$};
  \fill[black] (o) circle [radius=2pt] node [left] {0};

  \draw (0.72*-4,0) node[right] {$\Omega$};
  \draw [use tangent, decorate,
    decoration = {brace,mirror}] ($(o) + (-0.15,-0.05)$) 
    						--  ($(foot) + (-0.15,0.05)$)
    						node [pos=0.5, right] {$h(\Omega,x)$};

\end{tikzpicture}
    \caption{The length $h(\Omega, x)$.}
    \label{fig:h}
\end{figure}
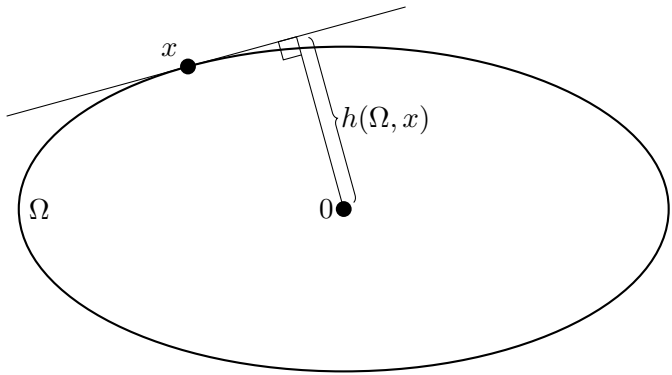

Before we can state our results, we must set up one piece of notation.  Let $\Omega \subseteq \mathbb{R}^d$ ($d \geq 2$) be a Jensen candidate such that $\partial \Omega$ is piecewise differentiable and $0$ lies in the interior of $\Omega$.  At each differentiable point $x \in \partial \Omega$, define $h(\Omega,x)$ to be the (orthogonal) distance from the origin to the tangent hyperplane to $\partial \Omega$ at $x$; this is the \emph{cone volume measure} on $\partial \Omega$ with respect to the origin.  Define $h_{\max}(\Omega)$ to be the supremum of $h(\Omega,x)$ over all differentiable points $x \in \partial \Omega$.  See Figure~\ref{fig:h}.  Notice that different translates of $\Omega$ can have different values of $h_{\max}$.  The following is our main result.

\begin{theorem}\label{thm:main}
Let $\Omega \subseteq \mathbb{R}^d$ be a Jensen candidate such that $\partial \Omega$ is piecewise differentiable and $0$ lies in the interior of $\Omega$.  If $$h_{\max}(\Omega) \leq \frac{(d+1)|\Omega|}{|\partial \Omega|},$$ then $\Omega$ is a Jensen domain.
\end{theorem}

The piecewise differentiability condition on $\partial \Omega$ can probably be relaxed.  As written, though, it is mild enough that we can recover Pasteczka's Theorem~\ref{thm:pasteczka} by noting that if $\Omega \subseteq \mathbb{R}^d$ is a convex polytope with an inscribed ball centered at the origin, then $h_{\max}(\Omega)=d|\Omega|/|\partial \Omega|$ (in fact $h(\Omega, \cdot)$ equals this quantity at all differentiable points).

Likewise, when $\Omega$ is a $d$-dimensional ball centered at the origin, we have $h_{\max}(\Omega)=d|\Omega|/|\partial \Omega|$.  Notice that the quantity $h_{\max}(\Omega)$ varies continuously as $\Omega$ is deformed.  It follows that the inequality $h_{\max}(\Omega) \leq (d+1)|\Omega|/|\partial \Omega|$ also holds if $\Omega$ is a perturbation of a ball centered at the origin, i.e., $\Omega$ is ``nearly spherical'', or ``not too oblong''.  The following corollary gives one way of making this observation precise.

\begin{corollary}\label{cor:spherical-shell}
Let $\Omega \subseteq \mathbb{R}^d$ ($d \geq 2$) be a Jensen candidate such that $\partial \Omega$ is piecewise differentiable and $0$ lies in the interior of $\Omega$.  If $\partial \Omega$ is contained in the spherical shell $\{x: 1 \leq |x| \leq (1+1/d)^{1/d}\}$, then $\Omega$ is a Jensen domain.
\end{corollary}

(The constant $(1+1/d)^{1/d}$ can be improved somewhat, but we have not attempted to optimize it in our argument.)

When $h_{\max}(\Omega)<(d+1)|\Omega|/|\partial \Omega|$, we can characterize the equality cases of \eqref{eq:main}.

\begin{theorem}\label{thm:equality}
Let $\Omega \subseteq \mathbb{R}^d$ ($d \geq 2$) be a Jensen candidate such that $\partial \Omega$ is piecewise differentiable, $0$ lies in the interior of $\Omega$, and $$h_{\max}(\Omega) < \frac{(d+1)|\Omega|}{|\partial \Omega|}.$$  Let $f: \Omega \to \mathbb{R}$ be a continuous convex function.
Then equality holds in \eqref{eq:main} if and only if $f$ is an affine-linear function.
\end{theorem}

It remains an open problem to find further conditions that are necessary or sufficient for a domain to be Jensen.

\section{Proofs}\label{sec:proofs}
Our proof of Theorem~\ref{thm:main} follows the same general proof strategy as Pasteczka's proof of Theorem~\ref{thm:pasteczka}.  Pasteczka (essentially) treated the special case where the function $h(\Omega, \cdot)$ is constant and equal to $d|\Omega|/|\partial \Omega|$.  The main idea of the following lemma is to calculate $\int_\Omega f \,dV$ using a spherical coordinate system.

\begin{lemma}\label{lem:polar}
Let $\Omega \subseteq \mathbb{R}^d$ ($d \geq 2$) be a compact convex domain such that $\partial \Omega$ is piecewise differentiable and $0$ lies in the interior of $\Omega$.  If $f:\Omega \to \mathbb{R}$ is a nonnegative convex function such that $f(0)=0$, then
$$\int_\Omega f \,dV \leq \frac{h_{\max}(\Omega)}{d+1} \int_{\partial \Omega} f \,d\sigma.$$
\end{lemma}

\begin{proof}
Notice that $\Omega$ is radially convex with respect to the origin.  Hence, each point of $\Omega$ other than the origin can be uniquely written as $tx$ where $x \in \partial \Omega$ and $t \in (0,1]$.  In this $(x,t)$-coordinate system for $\Omega$, the volume element is $dV=t^{d-1}h(\Omega,x) \, d\sigma(x) \,dt$, where $\sigma$ is the surface measure on $\partial \Omega$; this is a routine Jacobian calculation.  Since $f$ is convex and vanishes at the origin, we have
$$f(x,t) \leq (1-t)f(x,0)+tf(x,1)=tf(x,1).$$
Using Fubini's Theorem, the above inequality, and the nonnegativity of $f$, we get that
\begin{align*}
\int_\Omega f \,dV &= \int_{\partial \Omega} \int_0^1 f(x,t) t^{d-1}h(\Omega,x) \,d\sigma(x) \,dt\\
 &\leq \int_{\partial \Omega} \int_0^1 tf(x,1)t^{d-1}h(\Omega,x) \,d\sigma(x) \,dt\\
  &=\int_{\partial \Omega} f(x,1)h(\Omega,x) \,d\sigma(x) \int_0^1 t^d \,dt\\
  & \leq \frac{h_{\max}(\Omega)}{d+1} \int_{\partial \Omega} f \,d\sigma,
\end{align*}
as desired.
\end{proof}

An alternative perspective on this calculation (suggested by Ramon van Handel) comes from applying the Divergence Theorem to the function $x f(x)$.  With this lemma in hand, we can quickly deduce Theorem~\ref{thm:main}.

\begin{proof}[Proof of Theorem~\ref{thm:main}]
Let $f: \Omega \to \mathbb{R}$ be a convex function.  Since $f$ is convex, there is an affine-linear function $g: \Omega \to \mathbb{R}$ such that $g(0)=f(0)$ and $g(x) \leq f(x)$ for all $x \in \Omega$.  If $f$ is differentiable at the origin, then $g(x)=f(0)+x \cdot \nabla f(x)$ is the equation of the hyperplane tangent to the graph of $f$ at $(0, f(0))$; otherwise $g$ need not be unique.  Now $\tilde{f}:=f-g$ is a nonnegative convex function on $\Omega$ satisfying $\tilde{f}(0)=0$.  Lemma~\ref{lem:polar} tells us that
$$\int_\Omega \tilde{f} \,dV \leq \frac{h_{\max}(\Omega)}{d+1} \int_{\partial \Omega} \tilde{f} \,d\sigma,$$
and hence (by the assumption on $h_{\max(\Omega)}$)
$$\frac{1}{|\Omega|}\int_\Omega \tilde{f} \,dV \leq \frac{1}{|\partial \Omega|} \int_{\partial \Omega} \tilde{f} \,d\sigma.$$
Since $g$ is affine-linear and $\Omega$ is a Jensen candidate, we have
$$\frac{1}{|\Omega|}\int_\Omega g \,dV = \frac{1}{|\partial \Omega|} \int_{\partial \Omega} g\,d\sigma.$$
Adding the preceding two equations gives
$$\frac{1}{|\Omega|}\int_\Omega f \,dV \leq \frac{1}{|\partial \Omega|} \int_{\partial \Omega} f \,d\sigma.$$
Hence, $\Omega$ is a Jensen domain.
\end{proof}

Next, we prove Corollary~\ref{cor:spherical-shell}.

\begin{proof}[Proof of Corollary~\ref{cor:spherical-shell}]
Write $\lambda:=(1+1/d)^{1/d}$.  By Theorem~\ref{thm:main}, it suffices to show that $h_{\max}(\Omega) \leq (d+1)|\Omega|/|\partial \Omega|$.  Let $\omega_d$ denote the volume of the $d$-dimensional unit ball; recall that the surface area of the $d$-dimensional unit ball is $d\omega_d$.  Clearly $|\Omega| \geq \omega_d$.  We also have $|\partial \Omega| \leq \lambda^{d-1}d\omega_d$ since the nearest-point projection to $\Omega$ is distance-non-increasing and $\Omega$ is contained in a ball of radius $\lambda$.  Finally, $h(\Omega,x)$ is everywhere at most $\lambda$.  Putting everything together gives
$$\frac{(d+1)|\Omega|}{|\partial \Omega|} \geq \frac{(d+1)\omega_d}{\lambda^{d-1}d\omega_d}=\lambda \geq h_{\max}(\Omega),$$
as needed.
\end{proof}

We now turn to Theorem~\ref{thm:equality}.  We remark that the continuity assumption on $f$ in the statement of the theorem is very mild since every convex function on $\Omega$ is automatically continuous on the interior of $\Omega$. (Allowing discontinuities on $\partial \Omega$ can artificially increase the right-hand side of \eqref{eq:main}.)

\begin{proof}[Proof of Theorem~\ref{thm:equality}]
It is clear from the discussion in the introduction that equality holds in \eqref{eq:main} for every affine-linear function $f$.  Now, suppose $f$ satisfies \eqref{eq:main} with equality.  Then the inequalities in the proof of Theorem~\ref{thm:main} are all equalities.  In particular, since $h_{\max}(\Omega)<(d+1)|\Omega|/|\partial \Omega|$, comparing the first two inequalities in the proof of Theorem~\ref{thm:main} gives
$$\int_{\partial \Omega} \tilde{f} \,d\sigma=0.$$
Since $\tilde{f}$ is continuous and nonnegative on $\Omega$, we conclude that $\tilde{f}$ is uniformly zero on $\partial \Omega$.  Then $\tilde{f}$ is uniformly zero on all of $\Omega$ by convexity, and $f=g$ is affine-linear, as desired.
\end{proof}

\section*{Acknowledgements}
The first author is supported in part by an NSF Graduate Research Fellowship (grant DGE--2039656).  This research was conducted during the 2023 mathematics REU program at the University of Minnesota, Duluth (supported by Jane Street Capital and the National Security Agency); we thank Joe Gallian for providing this wonderful opportunity.  We thank Stefan Steinerberger for bringing Pasteczka's conjecture to our attention, and we thank him and Ramon van Handel for making valuable comments on an earlier draft of the paper.  Finally, we thank Maya Sankar for helping us make Figure~\ref{fig:h}.

\end{document}